\documentclass[12pt,a4paper]{amsart}

\usepackage[utf8]{inputenc}
\usepackage[T1]{fontenc}
\usepackage[normalem]{ulem}
\usepackage{cancel}
\usepackage{xcolor}

\usepackage[pdftex]{graphicx}
\usepackage{latexsym}
\usepackage{amsmath,amssymb,amsthm}
\usepackage[shortlabels]{enumitem}   
\usepackage{nicefrac} 
\allowdisplaybreaks 

\usepackage{scalerel,stackengine}
\stackMath
\newcommand\reallywidehat[1]{%
\savestack{\tmpbox}{\stretchto{%
  \scaleto{%
    \scalerel*[\widthof{\ensuremath{#1}}]{\kern-.6pt\bigwedge\kern-.6pt}%
    {\rule[-\textheight/2]{1ex}{\textheight}}
  }{\textheight}%
}{0.5ex}}%
\stackon[1pt]{#1}{\tmpbox}
}

\usepackage[Option]{overpic}
\usepackage{tikz}
\usetikzlibrary{positioning,arrows}
\tikzset{
  arrow/.style={-latex, shorten >=1ex, shorten <=1ex}}

\newtheorem*{Theorem*}{Theorem}
\newtheorem{Theorem}{Theorem}
\newtheorem*{Lemma*}{Lemma}
\newtheorem{Lemma}{Lemma}	 
\newtheorem{Corollary}{Corollary}	
\theoremstyle{definition}
\newtheorem{Remark}{Remark}	 
\newtheorem{Definition}{Definition} 
   
\newcommand{\C}{\mathbb{C}} 
\newcommand{\R}{\mathbb{R}} 
\newcommand{\Z}{\mathbb{Z}} 
\newcommand{\N}{\mathbb{N}} 

\newcommand{\RT}{\operatorname{Re}}
\newcommand{\IT}{\operatorname{Im}}
\newcommand{\iu}{\mathrm{i}}
\newcommand{\eu}{\mathrm{e}}
\newcommand{\eps}{\varepsilon}

\DeclareMathOperator{\Ker}{Ker}
\DeclareMathOperator{\Span}{Span}
\DeclareMathOperator{\Ran}{Ran}
\DeclareMathOperator{\Dom}{Dom}
\DeclareMathOperator{\sech}{sech}

\makeatletter\@addtoreset {equation}{section}\makeatother

\usepackage[colorlinks=true,urlcolor=blue,linkcolor=blue]{hyperref} 			

\begin{document}

\title[Stability of solitary waves in the LLE]{Stability of solitary wave solutions in the Lugiato-Lefever equation}

\author{Lukas Bengel}
\address{L. Bengel \hfill\break 
Institute for Analysis,\hfill\break
Karlsruhe Institute of Technology (KIT), \hfill\break
D-76128 Karlsruhe, Germany}
\email{lukas.bengel@kit.edu}

\begin{abstract}
We analyze the spectral and dynamical stability of solitary wave solutions to the Lugiato-Lefever equation (LLE) on $\R$. Our interest lies in solutions that arise through bifurcations from the phase-shifted bright soliton of the nonlinear Schr\"odinger equation (NLS). These solutions are highly nonlinear, localized, far-from-equilibrium waves, and are the physical relevant solutions to model Kerr frequency combs. We show that bifurcating solitary waves are spectrally stable when the phase angle satisfies $\theta \in (0,\pi)$, while unstable waves are found for angles $\theta \in (\pi,2\pi)$. Furthermore, we establish orbital asymptotical stability of spectrally stable solitary waves against localized perturbations. Our analysis exploits the Lyapunov-Schmidt reduction method, the instability index count developed for linear Hamiltonian systems, and resolvent estimates.
\end{abstract}


\date{\today} 
	


\maketitle

\section{Introduction}
Kerr frequency combs generated in an externally driven Kerr nonlinear microresonator are very promising devices in optical communications or frequency metrology, enabling, for instance, high-speed data transmission of up to 1.44 Tbit/s, cf.~\cite{PfeifleBrasch2014}. They are optical signals consisting of a multitude of equally spaced excited modes in frequency space and are modeled by stable highly localized stationary periodic solutions of the Lugiato-Lefever equation (LLE)
\begin{align}\label{LLE}
	\iu u_t = - d u_{xx} + (\zeta-\iu) u - |u|^2 u   + \iu f, \qquad (x,t) \in \R^2.
\end{align}
Here $u = u(x,t) \in \C$ is the field amplitude in the resonator, $d \not=0$ is the dispersion, $\zeta \in \R$ is the offset between the external forcing frequency and the resonant frequency in the resonator called detuning, and $f\in \R$ describes the pump power inside the resonator of the external forcing. A physical derivation of \eqref{LLE} can be found in \cite{LugiatoLefever1987}. From a mathematical point of view, LLE is a damped and driven nonlinear Schr\"odinger equation (NLS). Motivated by the promising applications of Kerr frequency combs, the existence of stationary solutions of LLE has received considerable attention. A plethora of stationary solutions have been found in numerical simulations
\cite{BarashenkovSmirnov1998,GaertnerTrocha2019,NozakiBekki1986,Parra-Rivas2018,Parra-Rivas2014,
Parra-Rivas2016} or have been constructed analytically
\cite{BarashenkovZemlyanaya1999,GaertnerReichel2020,Godey2017,KaupNewell1978,
MandelReichel2017,MiyajiOhnishi2010}.
The naturally associated question of their spectral as well as their nonlinear stability with respect to different types of perturbations has gained interest recently
\cite{BarashenkovSmirnov1996,DelceyHaragus2018Periodic,DelcyHaragus2018Instab,
FengStanislavova2021,GodeyBalakireva2014,HakkaevStanislavova2019,HaragusJohnson2021Lin,
HaragusJohnson2021Nonlin,HaragusJohnson2023,PerinetVerschueren2017,StanislavovaStefanov2018,
TerronesMcLaughlin1990}.
Whereas nonlinear stability can be obtained under general spectral stability assumptions, spectral stability analyses themselves rely on the specific structure of the solutions and typically employ similar methods as were used to construct them. So far, spectral stability has been obtained for periodic small amplitude solutions of \eqref{LLE} arising through a Turing bifurcation of a homogeneous rest state \cite{DelcyHaragus2018Instab}. The only spectral stability result of far-from-equilibrium solutions of \eqref{LLE} that the author is aware of is that in \cite{HakkaevStanislavova2019}, where solutions to \eqref{LLE} are constructed by bifurcation from the cnoidal wave solutions solving NLS equation on the torus. Stability results for explicitly available solitary wave solutions in the forced NLS equation without damping have been obtained in \cite{Barashenkov1991} and also recently in \cite{FengStanislavova2021}.
Here, we present the first spectral stability result of far-from-equilibrium \emph{soliton solutions} of the damped and driven NLS \eqref{LLE}. The solutions under consideration are highly nonlinear and arise by bifurcation from bright solitons in the NLS equation in the anomalous regime $d>0$. They satisfy the approximation formula
\begin{align}\label{eq:approx_form}
	u(x) \approx u_\infty + \sqrt{2\zeta} \sech\left( \sqrt{\frac{\zeta}{d}} x\right) \eu^{\iu \theta_0}, \qquad x \in \R,
\end{align}
where $u_\infty \in \C$ is a constant background due to the forcing in \eqref{LLE} and the angle $\theta_0$ is found by solving the equation $\cos\theta_0 = 2 \sqrt{2\zeta}/(\pi f)$. 
Although they are non-periodic solutions of \eqref{LLE}, their exponential localization makes them a valid approximation of a frequency comb, cf.~Figure~\ref{fig:Solitons}.
We emphasize that the approximation \eqref{eq:approx_form} is frequently used in the physics literature to approximate Kerr frequency combs \cite{Brasch,Herr,Wabnitz}, which facilitates (formal) computations.

Contributions of this paper are threefold. In Theorem~\ref{thm:ex}, we prove the existence of solitary wave solutions of \eqref{LLE} that verify the approximation \eqref{eq:approx_form}. Therefore we consider the LLE with dispersion rescaled to $d=1$ as a perturbation of the focusing NLS in the following sense:
\begin{align}\label{LLEperp}
	\iu u_t = - u_{xx} + \zeta u - |u|^2 u   + \eps\iu \Psi(u), \qquad (x,t) \in \R^2,
\end{align}
where
$$
	\Psi(u) = -u + f,
$$
and $\eps$ is the bifurcation parameter.
We show that solitary wave solutions $u = u(\eps)$ bifurcate at $\eps = 0$ from the rotated NLS soliton
$$
	\phi_{\theta_0}(x) :=  \sqrt{2 \zeta} \sech(\sqrt{\zeta}x) \eu^{\iu  \theta_0}.
$$

In Theorem~\ref{thm:spec_stab}, we prove spectral stability of the solitary waves satisfying $\sin\theta_0,\eps>0$ and spectral instability for all other sign configurations of $\eps,\sin\theta_0$. This result relies on a detailed analysis of the spectral problem, where the crucial point is to understand the behavior of small eigenvalues for $\eps$ small in the linearized problem.

Finally, in Theorem~\ref{thm:nonlin_stab} we prove nonlinear asymptotic orbital stability of spectrally stable solitary waves. For the linear estimates, we establish high frequency resolvent estimates for families of operators in Hilbert spaces, see Theorem~\ref{thm:ex_bdd_resolvent}, which we then use to prove linear stability. Indeed, the resolvent estimates are needed to overcome the problem that a Spectral Mapping Theorem for the non-sectorial operator arising in the linearized equation is a-priori not available. Nonlinear stability then follows as a corollary of the linear stability result.

\begin{Remark}
Existence of solitary waves bifurcating from the NLS soliton has already been proven in \cite{GasmiDiss} using the Crandall-Rabinowitz Theorem of bifurcation from a simple eigenvalue. Here, we use a different approach based on a Lyapunov-Schmidt reduction in parameter-dependent spaces. This is advantageous because it directly yields an expansion of the solution needed in the spectral stability analysis. Further, we believe that our approach is flexible enough for possible extensions to bifurcation problems with higher dimensional kernels. 
\end{Remark}

\begin{Remark}
There is a long list of literature on persistence and stability of solitary solutions for other variants of the perturbed NLS, cf.~
\cite{AlexeevaBarashenkov1999,BarashenkovZemlyanaya1999PhysReview,
KapitulaSandstede1998,PromislowKutz2000} and the references therein.
\end{Remark}

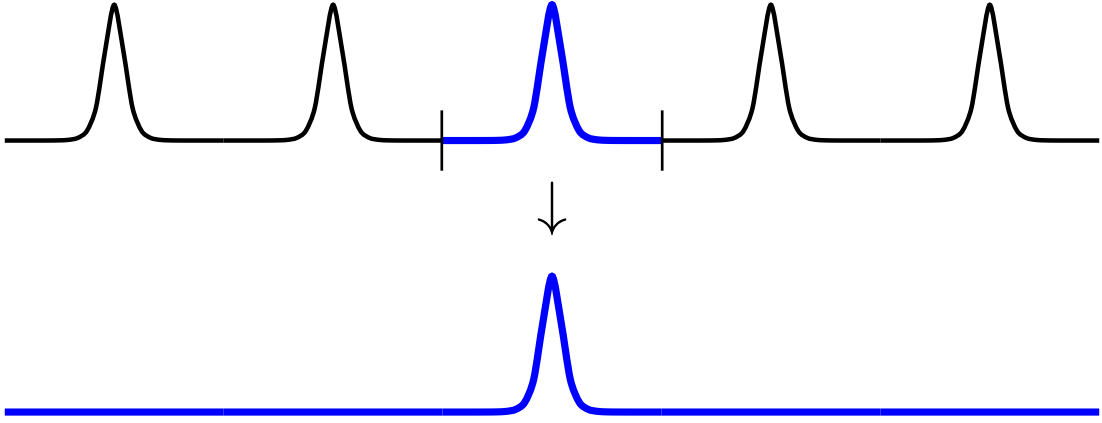
\begin{figure}[t]
\begin{center}
\begin{tikzpicture}
	\draw[scale=.6, domain=-8:8, smooth, variable=\x, blue, line width = 2.7pt] plot ({0.3*\x},{3*1/cosh(\x)*1/cosh(\x)+1});
	\draw[scale=.6, domain=-8:8, smooth, variable=\x, black, line width = 1.7pt] plot ({0.3*(\x-16)},{3*1/cosh(\x )*1/cosh(\x)+1});
	\draw[scale=.6, domain=-8:8, smooth, variable=\x, black, line width = 1.7pt] plot ({0.3*(\x+16)},{3*1/cosh(\x )*1/cosh(\x)+1});
	\draw[scale=.6, domain=-8:8, smooth, variable=\x, black, line width = 1.7pt] plot ({0.3*(\x+32)},{3*1/cosh(\x )*1/cosh(\x)+1});
	\draw[scale=.6, domain=-8:8, smooth, variable=\x, black, line width = 1.7pt] plot ({0.3*(\x-32)},{3*1/cosh(\x )*1/cosh(\x)+1});
	\draw[-, line width = 1pt] (-1.45,1) -- (-1.45,0.2);
	\draw[-, line width = 1pt] (1.45,1) -- (1.45,0.2);
	\draw (0,-.3) node [scale=1.9] {$\downarrow$};
	\draw[scale=.6, domain=-8:8, smooth, variable=\x, blue, line width = 2.7pt] plot ({0.3*\x},{3*1/cosh(\x )*1/cosh(\x)-5});
	\draw[scale=.6, domain=-8:8, smooth, variable=\x, blue, line width = 2.7pt] plot ({0.3*(\x-16)},{-5});
	\draw[scale=.6, domain=-8:8, smooth, variable=\x, blue, line width = 2.7pt] plot ({0.3*(\x-32)},{-5});
	\draw[scale=.6, domain=-8:8, smooth, variable=\x, blue, line width = 2.7pt] plot ({0.3*(\x+16)},{-5});
	\draw[scale=.6, domain=-8:8, smooth, variable=\x, blue, line width = 2.7pt] plot ({0.3*(\x+32)},{-5});
\end{tikzpicture}
\end{center}
\caption{Approximation of a periodic solution on $\R/\Z$ by a solitary wave on $\R$.}
\label{fig:Solitons}
\end{figure}

\subsection{Main results}\label{ssec:main_results}
Solitary wave solutions of the perturbed NLS \eqref{LLEperp} are solutions of the stationary equation
\begin{equation}\label{LLE_perp_stationary}
	 -u'' + \zeta u - |u|^2 u + \iu \eps (-u + f) = 0, \quad x \in \R,
\end{equation}
which decay to a limit state $u_\infty\in \C$ as $|x| \to \infty$.

The following theorem provides the first result of the paper on the existence of solitary waves for a suitable parameter region.

\begin{Theorem}\label{thm:ex}
Let $\zeta, f>0$ be fixed and suppose that $\theta_0 \in \R$ is a simple zero of the function
$$
	\theta \mapsto \pi f \cos\theta - 2\sqrt{2 \zeta}.
$$
Then there exist $\eps^*>0$ and a branch $(-\eps^*,\eps^*) \ni \eps \mapsto u(\eps) \in \C + H^2(\R)$ of solutions to the perturbed problem \eqref{LLE_perp_stationary} bifurcating from the rotated soliton
$$
	\phi_{\theta_0}(x) = \sqrt{2\zeta} \sech \big( \sqrt{\zeta} x \big) \eu^{\iu \theta_0}.
$$
More precisely, the branch is of the form
$$
	u (\eps) = \phi_{\theta(\eps)} + u_\infty(\eps) + \varphi(\eps),\quad u (0) = \phi_{\theta_0},
$$
where
\begin{itemize}
	\item the map $(-\eps^*,\eps^*) \ni\eps \mapsto \theta(\eps) \in \R$ is real-analytic and describes the rotational angle of the soliton,
	\item the map $(-\eps^*,\eps^*) \ni\eps \mapsto u_\infty(\eps) \in \C$ is real-analytic and consists of the constant background of the solution at $\pm \infty$,
	\item the map $(-\eps^*,\eps^*) \ni\eps \mapsto \varphi(\eps)\in H^2(\R)$ is real-analytic and describes a small correction term of order $\mathcal{O}(|\eps|)$. 
\end{itemize}
\end{Theorem}

\begin{Remark}
In Theorem~\ref{thm:ex}, a necessary and sufficient condition on the parameters $\zeta,f$ to find simple zeros is $\pi^2 f^2 > 8 \zeta$, which yields an existence region in the $\zeta$-$f$-plane, already obtained analytically or numerically in \cite{BarashenkovSmirnov1996,GaertnerReichel2020,GasmiDiss,Wabnitz}. Furthermore, it should be noted that solutions for negative forcing parameters $f<0$ are obtained through the transformation $u \mapsto - u$.
\end{Remark}

Theorem~\ref{thm:spec_stab} and Theorem~\ref{thm:nonlin_stab} below provide the two main results on the spectral and nonlinear stability of solitary waves of Theorem~\ref{thm:ex} against localized perturbations in $H^1$. 

Let $u = u(\eps) \in \C + H^2(\R)$ be a solution of the stationary LLE as in Theorem~\ref{thm:ex} for sufficiently small $\eps\not=0$. Expanding the solution as $\psi(x,t) = u(x) + v(x,t)$ results in the perturbation equation
\begin{align*}
	\iu v_t = - v_{xx} + \zeta v - 2 |u|^2 v - u^2 \bar{v} - \iu \eps v - 2 |v|^2 u - v^2 \bar{u} - |v|^2 v.
\end{align*}
The evolution of the perturbation $v$ is coupled with the evolution of the complex conjugate $\bar v$ so that we obtain the system:
\begin{align}\label{eq:dyn_pert}
	\left\{
	\begin{array}{rl}
		\iu v_t &= - v_{xx} + \zeta v - 2 |u|^2 v - u^2 \bar{v} - \iu \eps v - 2 |v|^2 u - v^2 \bar{u} - |v|^2 v, \\
	-\iu \bar v_t &= - \bar v_{xx} + \zeta\bar v - 2 |u|^2 \bar v -  \bar u^2 v + \iu \eps\bar v - 2 |v|^2\bar u - \bar v^2 u - |v|^2 \bar v,
	\end{array}\right.
\end{align}
for which the mild formulation is locally well-posed\footnote{This follows from standard semigroup theory.} in $(H^1(\R))^2$. The linearized equation is then given by
$$
	V_t
	= (\mathcal{L}-\eps)
	V, \quad V = (v_1,v_2)^T, 
$$
for the operator $\mathcal{L}:=J L: (H^2(\R))^2 \to (L^2(\R))^2$, with
\begin{align}\label{def_lin_op}
	J :=
	\begin{pmatrix}
		- \iu & 0 \\
		0 & \iu
	\end{pmatrix}, \quad L:=
	\begin{pmatrix}
		- \partial_x^2 + \zeta - 2 |u|^2 & - u^2 \\
		-\bar{u}^2 & - \partial_x^2 + \zeta - 2 |u|^2
	\end{pmatrix},
\end{align}
and the associated eigenvalue problem reads
\begin{align}\label{eq:spec_stab_prob}
	\lambda V = (\mathcal{L}-\eps) V, \quad V = (v_1,v_2)^T.
\end{align}
Spectral stability is now determined by the location of the spectrum of the linearized operator $\sigma(\mathcal{L}-\eps) = \sigma(\mathcal{L})-\eps$ according to the following definition.
\begin{Definition}
A solution $u = u(\eps) \in \C + H^2(\R)$ of \eqref{LLE_perp_stationary} is called \emph{spectrally unstable}, if there exists $\lambda \in \sigma(\mathcal{L}-\eps)$ such that $\RT(\lambda)>0$. Otherwise the solution is called \emph{spectrally stable}, i.e., if and only if $\sigma(\mathcal{L}) \subset\{z \in \C:\ \RT z \leq \eps\}$.
\end{Definition}

The following theorem clarifies the spectral stability of the solitary wave solutions found in Theorem~\ref{thm:ex}.

\begin{Theorem}\label{thm:spec_stab}
Suppose that $u = u(\eps) \in \C + H^2(\R)$ is a solution of \eqref{LLE_perp_stationary} as in Theorem~\ref{thm:ex} for $\eps\not=0$ sufficiently small, that bifurcates from the NLS soliton
$$
	\phi_{\theta_0}(x) = \sqrt{2\zeta} \sech \big( \sqrt{\zeta} x \big) \eu^{\iu \theta_0},
$$
with $\sin\theta_0 \not = 0$. Then, the spectrum of $\mathcal{L}-\eps$ is given by the disjoint union of essential and discrete spectrum $\sigma(\mathcal{L}-\eps) = \sigma_\textup{ess}(\mathcal{L}-\eps) \cup \sigma_d(\mathcal{L}-\eps)$, where the essential spectrum is explicitly computable:
$$
	\sigma_\textup{ess}(\mathcal{L}-\eps) = \left\{ \iu \omega \in \iu\R: \ |\omega| \in [\zeta_\eps,\infty) \right\}-\eps, \quad \zeta_\eps = \zeta + \mathcal{O}(\eps^2).
$$
Moreover,
\begin{enumerate}[(i)]
	\item if $\eps<0$, then the solution $u$ is spectrally unstable. The same is true if $\eps,-\sin\theta_0>0$.
	\item if $\eps,\sin\theta_0>0$, then the solution $u(\eps)$ is spectrally stable and the spectrum satisfies 
	$$
		\sigma(\mathcal{L}-\eps) \subset \{-2 \eps\} \cup \{z \in \C:\ \RT z = -\eps\} \cup \{0\}
	$$
	with (algebraically) simple eigenvalues $\lambda = 0,-2\eps$.
\end{enumerate}
\end{Theorem}

The different stability configurations of Theorem~\ref{thm:spec_stab} are depicted in Figure~\ref{fig1}.

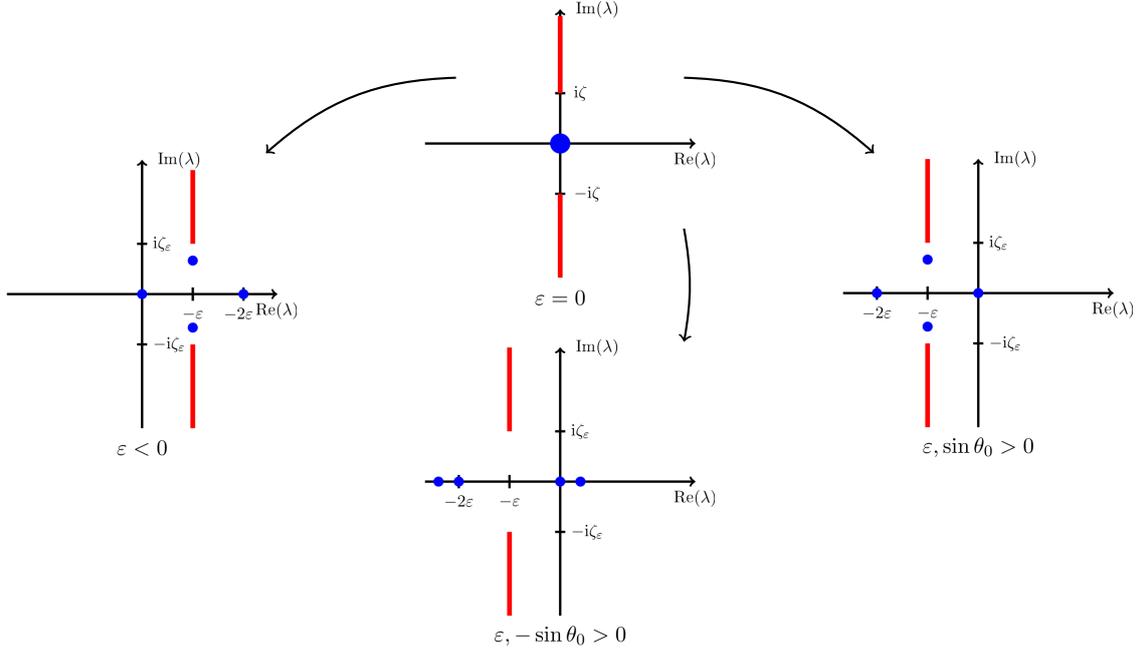
\begin{figure}[t]
\begin{tikzpicture}
\node at (0,3.5){\includegraphics[width=0.23\columnwidth]{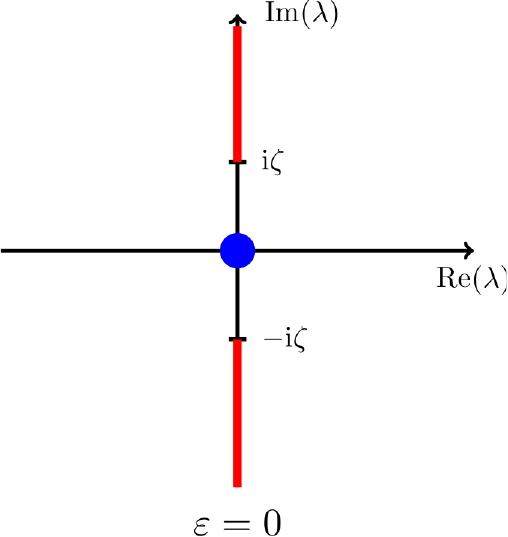}};
\node  at (0,-1) {\includegraphics[width=0.23\columnwidth]{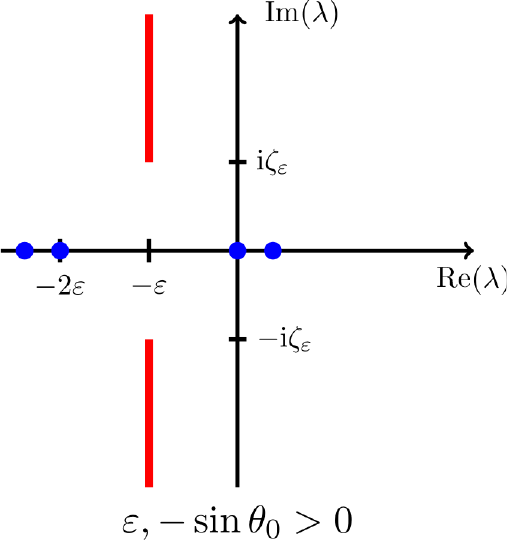}};
\node  at (-5.5,1.5) {\includegraphics[width=0.23\columnwidth]{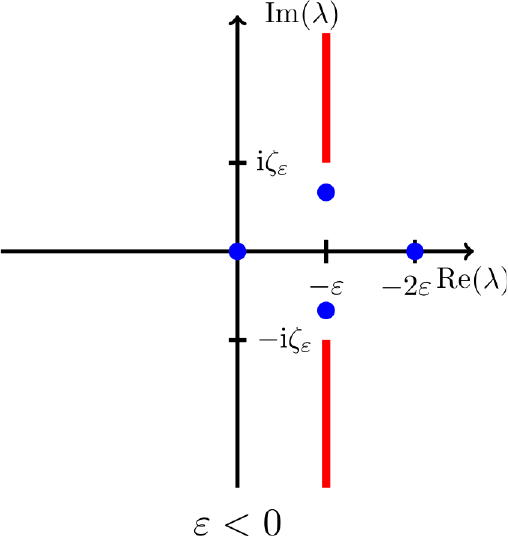}};
\node  at (5.5,1.5) {\includegraphics[width=0.23\columnwidth]{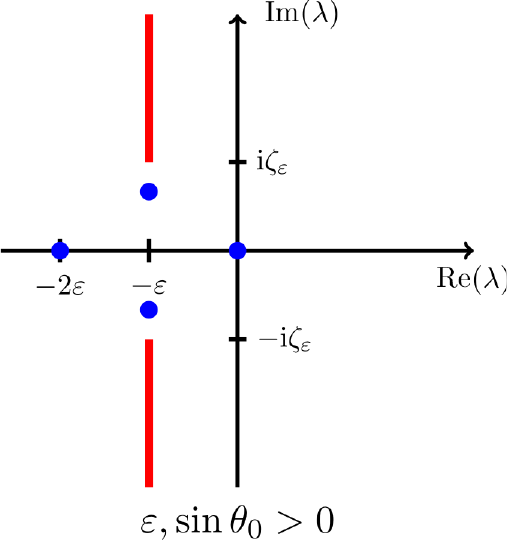}};
\draw [->,draw=black,line width = .8pt] (1.5,4.5) to[bend left=20] (4,3.5);
\draw [->,draw=black,line width = .8pt] (-1.5,4.5) to[bend right=20] (-4,3.5);
\draw [->,draw=black,line width = .8pt] (1.5,2.5) to[bend left=10] (1.5,1);
\end{tikzpicture}
\caption{Stability configurations of Theorem~\ref{thm:spec_stab}. Blue dots = discrete spectrum, red lines = essential spectrum. Top: spectrum of the unperturbed stable NLS soliton. Left and bottom: spectrum of unstable solitary waves of LLE. Right: spectrum of a stable solitary wave of LLE.}
\label{fig1}
\end{figure}

\begin{Remark}
In case (i) of Theorem~\ref{thm:spec_stab}, the instability is triggered by two different mechanism. If $\eps<0$, the essential spectrum of $\mathcal{L}-\eps$ is unstable. If $\eps>0>\sin\theta_0$, we find exactly one simple real unstable eigenvalue $\lambda_+ \in \sigma_d(\mathcal{L})$ of order $\mathcal{O}(\eps^{1/2})$. 
\end{Remark}

\begin{Remark}
A similar spectral stability result for periodic solutions of \eqref{LLEperp} bifurcating from cnoidal wave solutions of the NLS has been obtained in \cite{StanislavovaStefanov2018}. Moreover, in \cite{Barashenkov1991,FengStanislavova2021} stability and instability of purely imaginary soliton solutions of the forced NLS equation is proven. Here, the stable solutions have a strictly positive imaginary part, which is in agreement with our sign condition on $\sin\theta_0$.
\end{Remark}

The next theorem, provides the nonlinear stability result for spectrally stable solutions of Therorem~\ref{thm:spec_stab}.

\begin{Theorem}\label{thm:nonlin_stab}
Suppose that $u = u(\eps) \in \C + H^2(\R)$ is a spectrally stable solution as in Theorem~\ref{thm:spec_stab}. Then, the solution $u(\eps)$ is asymptotically orbitally stable.  More precisely, there exist constants $\delta,\eta,C>0$ such that for all $v_0 \in H^1(\R)$ satisfying $\| v_0\|_{H^1}< \delta$ there exists the unique global (mild) solution $(v,\bar{v}) \in C([0,\infty),(H^1(\R))^2)$, $(v(0),\bar{v}(0)) =(v_0,\bar{v}_0)$ of \eqref{eq:dyn_pert} and $\sigma_\infty \in \R$ such that with $\psi = u + v$ we have
$$
	\|\psi(\cdot,t) - u(\cdot - \sigma_\infty)\|_{H^1} \leq C \delta \eu^{-\eta t} \qquad \text{for $t \geq 0$.}
$$
\end{Theorem}

\begin{Remark}
In \cite{StanislavovaStefanov2018} asymptotic orbital stability of spectrally stable periodic solutions against co-periodic perturbations is proven. Theorem~\ref{thm:nonlin_stab} can be seen as an analogous version of this result for stability of solitary wave solutions on $\R$ against localized perturbations.
\end{Remark}

\begin{Remark}
If we restrict perturbations to the class of even functions the asymptotical orbital stability can be improved to asymptotic stability by exploiting a spectral gap in the linear stability problem.
\end{Remark}

\subsection{Outline of the paper}
In Section~\ref{sec:existence}, we show that solitary waves for the LLE bifurcate from the NLS soliton as stated in Theorem~\ref{thm:ex}.
The spectral stability problem is analyzed in Section~\ref{sec:spec_stability} and the proof of Theorem~\ref{thm:spec_stab} is presented.
The asymptotic orbital stability result of Theorem~\ref{thm:nonlin_stab} is proven in Section~\ref{sec:proof_nonlin_stab}. It relies upon uniform resolvent estimates for the linearized LLE, which we derive from high frequency resolvent estimates. These estimates are obtained in an abstract functional analytical set-up and can therefore also be applied to other NLS type equations.

\section{Existence of solitary wave solutions}\label{sec:existence}

The goal of this section is to prove Theorem~\ref{thm:ex}. We work in the Sobolev spaces $H^k(\R) = H^k(\R,\C)$, $k \in \N_0$ of complex valued functions \emph{over the field} $\R$. By this choice of function spaces, the map $H^2(\R) \ni u \mapsto |u|^2 u \in H^2(\R)$ is Fr\'echet differentiable. Moreover, the $L^2$ scalar-product is defined by the \emph{$\R$-valued map}
$$
	\forall f,g \in L^2(\R,\C): \quad\langle f, g \rangle_{L^2} = \RT \int_\R f \bar{g} dx.
$$
Let us fix the parameters $\zeta,f>0$ such that $2 \sqrt{2 \zeta} < \pi f$ and let $\theta_0 \in \R$ be a simple zero of $\theta \mapsto \pi f\cos \theta - 2\sqrt{2\zeta}$. Recall that $\phi_0(x) = \sqrt{2 \zeta} \sech \big( \sqrt{\zeta} x \big)$ denotes the soliton solution of the defocusing cubic NLS equation
$$
	-\phi'' + \zeta \phi - |\phi|^2 \phi = 0,
$$
which decays exponentially to zero as $|x| \to \infty$. Let us introduce the manifold of rotated solitons
$$
	\mathcal{M} := \Big\{\phi_\theta = \phi_0 \eu^{\iu \theta} : \theta \in \R \Big\}.
$$
By the gauge invariance of NLS every $\phi \in \mathcal{M}$ is a solution of the NLS. In addition, NLS possesses a translational symmetry, i.e., $\phi(\cdot-\sigma)$ is a solution of NLS for all shifts $\sigma \in \R$. In the Lugiato-Lefever equation, the gauge symmetry is broken, and only the translational symmetry persists. Consequently, the continuation for $\eps \not=0$ is expected to be successful only for suitably rotated solitons $\phi \in \mathcal{M}$. To handle the translational symmetry, we fix the shift parameter $\sigma = 0$ and restrict our analysis to the spaces $H_\text{ev}^2(\R),L_\text{ev}^2(\R)$ of even functions.

It is now important to note that the presence of the forcing term $f$ in \eqref{LLE_perp_stationary} prevents solutions from decaying to zero as $|x| \to \infty$. More precisely, the tails at $\pm \infty$ of every localized solution satisfy the algebraic equation
\begin{align}\label{eq:background}
	\zeta u_\infty - |u_\infty|^2 u_\infty + \iu\eps\big( - u_\infty + f) = 0, \quad u_\infty \in \C.
\end{align}
Thus, we adapt an ansatz of the form
$$
	u = \phi_\theta + u_\infty + \varphi,
$$
where $\phi_\theta \in \mathcal{M}$ is a suitably rotated soliton, $u_\infty \in \C$ is the constant background solving the algebraic equation \eqref{eq:background}, and $\varphi \in H_\text{ev}^2(\R)$ is a small correction. For the background we solve ~\eqref{eq:background} to leading order by
\begin{align}\label{eq:expansion_background}
	u_\infty(\eps) = -\frac{\iu f}{\zeta} \eps + \mathcal{O}(\eps^2),
\end{align}
which is the unique solution\footnote{There are two additional solutions to \eqref{eq:background} given by $u_\infty(\eps) = \pm \sqrt{\zeta}+\mathcal{O}(\eps)$, which are not considered here.} close to $0$. Inserting our ansatz into \eqref{LLE_perp_stationary} yields the equation for the correction $\varphi$ and rotational angle $\theta$:
\begin{align}\label{eq:correction}
	L_\theta \varphi = N(\varphi, \eps, \theta)
\end{align}
where
\begin{align*}
	L_\theta \varphi := -\varphi'' + \zeta \varphi - 2 |\phi_\theta|^2 \varphi - \phi_\theta^2 \bar\varphi, \qquad \varphi \in H_\text{ev}^2(\R),
\end{align*}
and
\begin{align*}
	N(\varphi, \eps, \theta) := & 2 |u_\infty(\eps) + \varphi|^2 \phi_\theta  +(u_\infty(\eps) + \varphi)^2 \bar\phi_\theta + |u_\infty(\eps) + \varphi|^2 (u_\infty(\eps) + \varphi) \\
	&- |u_\infty(\eps)|^2 u_\infty(\eps)
	+2 |\phi_\theta|^2 u_\infty(\eps) + \phi_\theta^2 \bar u_\infty(\eps) 
	+ \iu \eps (\phi_\theta + \varphi).
\end{align*}
Note that we have $N(\varphi,\eps,\theta) \in H_\text{ev}^2(\R)$ for $\varphi \in H_\text{ev}^2(\R)$. Our goal is to solve \eqref{eq:correction} by means of the Lyapunov-Schmidt reduction method. Therefore, let us collect relevant properties of the linearized operator $L_\theta$.

\begin{Lemma}\label{lem:ker(L)}
For every $\theta \in \R$ the $\R$-linear operator $L_\theta: H^2(\R) \to L^2(\R)$ is self-adjoint and Fredholm of index zero. Moreover we have
$$
	\Ker(L_\theta) = \Span\{ \iu \phi_\theta , \phi_\theta' \}.
$$
\end{Lemma}
\begin{proof}
The fact that $L_\theta$ is Fredholm of index zero follows from $0 \not \in \sigma_\text{ess}(L_\theta) = \sigma_\text{ess}(-\partial_x^2 + \zeta) = [\zeta,\infty)$ and the first equality is a consequence of Weyl's~Theorem. Moreover the operator
$ -\partial_x^2: H^2(\R) \to L^2(\R)$ is self-adjoint and hence the same holds for $L_\theta$ since it is a symmetric bounded perturbation of $-\partial_x^2:H^2(\R)\to L^2(\R)$. Finally, the identity for the kernel follows from a similarity transformation and explicit formulas for the kernel of the linearized NLS operator, cf.~\cite{KapProm}.
\end{proof}

Since we aim to solve \eqref{eq:correction} in the space of even functions, we note that the restricted operator $L_\theta|_{H_\text{ev}^2}$ has a one-dimensional kernel. Indeed, by Lemma~\ref{lem:ker(L)} the kernel is explicitly given by $\Ker(L_\theta|_{H_\text{ev}^2}) =  \Span\{ \iu\phi_\theta \}$. Lyapunov-Schmidt reduction now relies on the decomposition of $L_\text{ev}^2$ with the orthogonal projection onto $\Ker(L_\theta|_{H_\text{ev}^2})$ defined by
$$
	P_\theta \varphi : = \frac{\langle \iu\phi_\theta , \varphi \rangle_{L^2}}{\|\phi_0\|_{L^2}^2} \iu\phi_\theta, \qquad \varphi \in H_\text{ev}^2(\R).
$$
The operator $P_\theta$, allows us to split the equation \eqref{eq:correction} into a singular and non-singular part:
\begin{align}
	(I-P_\theta) L_\theta(I-P_\theta) \varphi &= (I-P_\theta) N(\varphi,\eps,\theta) \label{eq:LS_nonsing}\\
	P_\theta N(\varphi,\eps,\theta) &= 0 \label{eq:LS_sing}
\end{align}
where we additionally impose the phase-condition:
\begin{align}\label{eq:phase_cond}
	P_\theta \varphi = 0.
\end{align}
Note that the condition \eqref{eq:phase_cond} depends on the free rotational parameter $\theta$ which means that our decomposition of $L_\text{ev}^2$ is parameter dependent:
$$
	L_\text{ev}^2(\R) = \Ker(P_\theta) \oplus \Ran(P_\theta).
$$
In the following lemma we solve the non-singular equation \eqref{eq:LS_nonsing} subject to the phase-condition \eqref{eq:phase_cond}.

\begin{Lemma}
\label{lem:nonsing}
There exist open neighborhoods $U \subset \R^2$ of $(0,\theta_0)$, $V \subset H_\textup{ev}^2(\R)$ of $0$ and an real-analytic map $U \ni (\eps,\theta) \mapsto \varphi(\eps,\theta) \in V$ such that $\varphi(\eps,\theta)$ solves \eqref{eq:LS_nonsing} subject to the phase condition \eqref{eq:phase_cond} and $\varphi(0,\theta) = 0$ for all $(0,\theta) \in U$.
\end{Lemma}

\begin{proof}
Define the function $F: H_\text{ev}^2(\R) \times \R \times \R \to  L_\text{ev}^2(\R)$ given by
$$
	F(\varphi,\eps,\theta):= (I-P_\theta) L_\theta(I-P_\theta) \varphi- (I-P_\theta) N(\varphi,\eps,\theta) + P_\theta \varphi.
$$
Then, $F$ is real-analytic in $(\varphi,\eps,\theta)$ as a composition of real-analytic functions (cf.~\cite{Buffoni_Toland} Theorem~4.5.7), $F(0,0,\theta_0) = 0$ and
$$
	\partial_{\varphi}F(0,0,\theta_0) = (I-P_{\theta_0}) L_\theta(I-P_{\theta_0}) + P_{\theta_0}:H_\text{ev}^2(\R) \to L_\text{ev}^2(\R)
$$
is a homeomorphism by construction. The Implicit Function Theorem for analytic functions (\cite{Buffoni_Toland} Theorem~4.5.4) yields the existence of open neighborhoods $U \subset \R^2$ of $(0,\theta_0)$, $V \subset H_\text{ev}^2(\R)$ of $0$ and a real-analytic map $U \ni (\eps,\theta) \mapsto \varphi(\eps,\theta) \in V$ such that the unique solution of $F(\varphi,\eps,\theta) = 0$ in $V \times U$ is given by $(\varphi, \eps, \theta) = (\varphi(\eps,\theta), \eps, \theta)$. Finally, since $F(0,0,\theta) = 0$, we find $\varphi(0,\theta) = 0$ by the local uniqueness of the solution.
\end{proof}
Substitution of the solution obtained in Lemma~\ref{lem:nonsing} into \eqref{eq:LS_sing} amounts to
$$
	P_\theta N(\varphi(\eps,\theta),\eps,\theta) = 0 \quad\Longleftrightarrow\quad
	f(\eps,\theta):=\langle N(\varphi(\eps,\theta),\eps,\theta), \iu\phi_\theta \rangle_{L^2} = 0,
$$
where $f: U \subset \R^2 \to \R$ is again real-analytic as a composition of real-analytic functions and admits the expansion
$$
	f(\eps,\theta) = \langle \eps\iu \phi_\theta + 2 |\phi_\theta|^2 u_\infty(\eps) + \phi_\theta^2  \bar u_\infty(\eps)
	, \iu\phi_\theta \rangle_{L^2} 
	+ \mathcal{O}(\eps^2).
$$
Clearly, $f(0,\theta) = 0$ for all $(0,\theta) \in U$ and thus we find a real-analytic function $\tilde{f}$ with $f(\eps,\theta) = \eps \tilde{f}(\eps,\theta)$. Nontrivial solutions of $f(\eps,\theta) = 0$ then satisfy $\tilde{f}(\eps,\theta) = 0$ and the equation is again solved by the Implicit Function Theorem. Indeed, in the subsequent Lemma~\ref{lem:formula_for_f} we show $\tilde{f}(0,\theta_0) = 0$ and $\partial_\theta\tilde{f}(0,\theta_0) \not = 0$ and thus there exist open intervals $(-\eps^*,\eps^*)$, $\Theta \subset \R$, $\theta_0 \in \Theta$ and a real-analytic branch $(-\eps^*,\eps^*) \ni\eps \mapsto\theta(\eps) \in \Theta$ such that $\tilde{f}(\eps,\theta) = 0$ in $(-\eps^*,\eps^*) \times \Theta$ is uniquely solved by $(\eps,\theta(\eps)) = (\eps,\theta)$. In summary we have constructed a solitary wave solution 
$$
	u (\eps) = \phi_{\theta(\eps)} + u_\infty(\eps) + \varphi(\eps,\theta(\eps)) \in \C + H^2(\R), \quad\eps \in (-\eps^*,\eps^*)
$$
of the Lugiato-Lefever equation \eqref{LLE_perp_stationary}. It remains to prove Lemma~\ref{lem:formula_for_f}.

\begin{Lemma}\label{lem:formula_for_f}
Let $\theta_0 \in \R$ be a simple zero of $\theta \mapsto \pi f\cos \theta - 2\sqrt{2\zeta}$. Then $\tilde{f}(0,\theta_0) = 0$ and $\partial_\theta\tilde{f}(0,\theta_0) \not = 0$.
\end{Lemma}

\begin{proof}
By definition of $\tilde{f}$, formula \eqref{eq:expansion_background}, and since $\pi f\cos\theta_0 = 2 \sqrt{2\zeta}$, we have
\begin{align*}
	\tilde{f}(0,\theta_0) &=  \langle \iu \phi_{\theta_0} + 2 |\phi_{\theta_0}|^2 \partial_\eps u_\infty(0) + \phi_{\theta_0}^2  \partial_\eps\bar u_\infty(0)
	, \iu\phi_{\theta_0} \rangle_{L^2} \\
	&= \RT \int_\R |\phi_0|^2 - 2 \iu |\phi_0|^2 \bar \phi_{\theta_0} \partial_\eps u_\infty(0) - \iu |\phi_0|^2 \phi_{\theta_0} \partial_\eps\bar u_\infty(0)  dx \\
	&= 4 \sqrt{\zeta} - \frac{f}{\zeta} \RT \int_\R 2 |\phi_0|^2 \bar \phi_{\theta_0} -  |\phi_0|^2 \phi_{\theta_0}  dx \\
	&= 4 \sqrt{\zeta} - \frac{f}{\zeta} \cos\theta_0 \int_\R \phi_0^3 dx  \\
	&= 4  \sqrt{\zeta} - \sqrt{2}\pi f \cos\theta_0 =  0
\end{align*}
and similar computations lead to
\begin{align*}
	\partial_\theta\tilde{f}(0,\theta_0) = \sqrt{2} \pi f \sin\theta_0 \not = 0,
\end{align*}
where $\sin\theta_0\not =0$ by simplicity of the root $\theta_0$, which proves the statement.
\end{proof}

\begin{Corollary}\label{cor:exp_conv}
The solutions $u (\eps) \in \C + H^2(\R)$ of Theorem~\ref{thm:ex} decay exponentially fast to their limit states $u_\infty(\eps) \in \C$ as $|x| \to \infty$.
\end{Corollary}
\begin{proof}
Re-writing \eqref{LLE_perp_stationary} in its dynamical system formulation
$$
	\partial_x U=
	\begin{pmatrix}
		U_3 \\ U_4 \\ \zeta U_1 - (U_1^2+U_2^2)U_1 + \eps U_2 \\
		\zeta U_2 - (U_1^2+U_2^2)U_2 + \eps (-U_1 + f)
	\end{pmatrix}, \quad
	U =
	\begin{pmatrix}
		u_1(\eps) \\ u_2(\eps) \\ \partial_x u_1(\eps) \\ \partial_x u_2(\eps)
	\end{pmatrix}		
$$
for $u_1(\eps) = \RT(u(\eps)), u_2(\eps) = \IT(u(\eps))$ we easily see that $U$ is homoclinic to a hyperbolic equilibrium and thus converges exponentially fast to its limit state 
$
	U_\infty = (\RT (u_{\infty}(\eps)), \allowbreak \IT (u_{\infty}(\eps)), 0, 0)^T.
$
\end{proof}

\section{Stability analysis}\label{sec:stability}

In this section, we prove the stability results of Theorem~\ref{thm:spec_stab}~and Theorem~\ref{thm:nonlin_stab}. From now on $H^k(\R)=H^k(\R, \C)$, $k \in \N_0$ denotes the Sobolev spaces \emph{over the field} $\C$. In particular, the $L^2$-scalar product is now given by the \emph{$\C$-valued map}
$$
	\forall f,g \in L^2(\R,\C): \quad \langle f, g \rangle_{L^2} = \int_\R f \bar g dx.
$$
Let us start with the proof of the spectral stability result of Theorem~\ref{thm:spec_stab}.

\subsection{Proof of Theorem~\ref{thm:spec_stab}}\label{sec:spec_stability}
Suppose that $u = u(\eps) \in \C+ H^2(\R)$ is a solution of \eqref{LLE_perp_stationary} as in Theorem~\ref{thm:spec_stab} for $\eps\not =0$ sufficiently small. We determine the location of the spectrum of the operator $\mathcal{L}=JL$ defined in \eqref{def_lin_op}.

It is well-known \cite{KapProm} that we have a decomposition into essential and discrete spectrum
\begin{align}\label{eq:decomp_spec}
	\sigma(\mathcal{L}) = \sigma_\text{ess}(\mathcal{L}) \cup \sigma_d(\mathcal{L}).
\end{align}

\subsubsection*{\bf Essential spectrum of $\mathcal{L}$} The essential spectrum can be computed explicitly.

\begin{Lemma}\label{lem:ess_spec}
Let $\eps$ be sufficiently small. The essential spectrum is given by
$$
	\sigma_\text{ess}(\mathcal{L}) = \{\iu \omega \in \iu \R : \ |\omega| \in [\zeta_\eps, \infty) \},
$$
where $\zeta_\eps = \zeta + \mathcal{O}(\eps^2)$.
\end{Lemma}

\begin{proof}
Since $u(x) \to u_\infty$ as $|x| \to \infty$ exponentially fast, cf.~Corollary~\ref{cor:exp_conv}, we can use Weyl's Theorem~\cite{KapProm} to find
$$
	\sigma_\text{ess}(\mathcal{L}) = \sigma_\text{ess}(\mathcal{L}^\infty)
$$
where the asymptotic constant coefficient operator is given by
$$
	\mathcal{L}^\infty = 
	\begin{pmatrix}
		- \iu & 0 \\
		0 & \iu
	\end{pmatrix}
	\begin{pmatrix}
		- \partial_x^2 + \zeta - 2 |u_\infty|^2 & - u_\infty^2 \\
		-\bar{u}_\infty^2 & - \partial_x^2 + \zeta - 2 |u_\infty|^2
	\end{pmatrix}.
$$
We calculate the essential spectrum of $\mathcal{L}^\infty $. Therefore, we write the spectral problem for $\mathcal{L}^\infty$ in its first order reformulation $\partial_x V = A(\lambda)V$ with
$$
	A(\lambda) =
	\begin{pmatrix}
		0 & 0 & 1 & 0 \\
		0 & 0 & 0 & 1 \\
		\zeta - 2|u_\infty|^2-\iu \lambda & u_\infty^2 & 0 & 0 \\
		\bar{u}_\infty^2 & \zeta - 2|u_\infty|^2+\iu \lambda & 0 & 0
	\end{pmatrix}, \quad
	V = 
	\begin{pmatrix}
		v_1 \\ v_2 \\ \partial_x v_1 \\ \partial_x v_2
	\end{pmatrix}.
$$
Then, the essential spectrum is characterized as follows (see for instance~\cite{KapProm}):
$$
	\lambda \in \sigma_\text{ess}(\mathcal{L}^\infty) \quad \Longleftrightarrow \quad
	\sigma(A(\lambda))\cap \iu \R \not = \emptyset.
$$
Computing
$$
	p(k,\lambda):=\det(A(\lambda)- \iu k) = k^4 + 2( \zeta-2|u_\infty|^2 ) k^2 + (\zeta-2|u_\infty|^2)^2-|u_\infty|^4   +\lambda^2 
$$
we find that $A(\lambda)$ is nonhyperbolic if and only if
$$
	\exists k \in \R: \quad p(k,\lambda)=0
$$
which is equivalent to
$$
	\lambda \in\{ \iu \omega \in \iu\R : \ |\omega| \in  [\zeta_\eps, \infty) \},
$$
with $\zeta_\eps = \zeta + \mathcal{O}(\eps^2)$ since $u_\infty = \mathcal{O}(\eps)$ and thus the claim follows.
\end{proof}

An immediate consequence of Lemma~\ref{lem:ess_spec} is the spectral instability of the wave $u$ if $\eps <0$. However, if $\eps>0$, the essential spectrum is stable and spectral stability is solely determined by the location of the discrete spectrum. From now on we focus on the case $\eps>0$.

\subsubsection*{\bf Discrete spectrum of $\mathcal{L}$}  Recall, that the discrete spectrum can be determined from the eigenvalue problem \eqref{eq:spec_stab_prob} which can be written as
\begin{align*}
	\left\{
	\begin{array}{rl}
		\iu \lambda v_1 &= - v_{1xx} + \zeta v_1 - 2 |u|^2 v_1 - u^2 v_2 - \iu \eps v_1, \\
		-\iu \lambda v_2 &= - v_{2xx} + \zeta v_2 - 2 |u|^2 v_2 - \bar u^2 v_1 + \iu \eps v_2.
	\end{array}\right.
\end{align*}
For $\eps=0$, we recover the spectral stability problem for the rotated soliton $\phi_{\theta_0}$ of NLS and $\lambda=0$ is an isolated eigenvalue of geometric multiplicity two and algebraic multiplicity four. More precisely, we find two Jordan chains of length two and by Lemma~\ref{lem:ker(L)} we have that the corresponding eigenspace is spanned by the vectors $(\iu \phi_{\theta_0},-\iu \bar{\phi}_{\theta_0})$ and $(\phi_{\theta_0}',\bar{\phi}_{\theta_0}')$. Consequently, it follows from standard perturbation theory~\cite{Kato}, that for small values of $\eps$ the total multiplicity of all eigenvalues in a small neighborhood of zero is also four. We now focus on the bifurcations of these eigenvalues which in the end will determine the spectral stability.

\subsubsection*{Eigenvalues close to the origin}
We compute expansions in $\eps$ of all eigenvalues of $\mathcal{L}=JL$ close to zero.
Observe that we always find $0 \in \sigma(\mathcal{L}-\eps)$ due to the translational invariance of \eqref{LLE_perp_stationary}. This yields $\eps \in \sigma(JL)$ and the corresponding eigenfunction is given by $(\partial_x u, \partial_x \bar{u})$. Since the spectrum of $JL$ is symmetric w.r.t.~the imaginary axis we also find $-\eps \in \sigma(JL)$. Note that the symmetry is an immediate consequence of the structure of $JL$, a composition of a skew-adjoint and self-adjoint operator. Thus, in the neighborhood of $\lambda=0$ only two unknown eigenvalues remain and they correspond to the broken gauge-symmetry in the LLE.

To compute the expansions of the perturbed rotational eigenvalues, we restrict to spaces of even functions. Following \cite{Kato}, we expand both remaining eigenvalues along with their corresponding eigenfunctions in a Puiseux series:
$$
	\lambda = \sqrt{\eps} \lambda_1+ \eps \lambda_2 + \mathcal{O}(\eps^{3/2}), \quad 
	V =V_0 + \sqrt{\eps} V_1 + \eps V_2 + \mathcal{O}(\eps^{3/2}),\quad V_0 = \begin{pmatrix} \iu \phi_{\theta_0} \\ - \iu \bar\phi_{\theta_0} \end{pmatrix},
$$
and the expansions are in powers of $\eps^{1/2}$ since we consider the splitting of a Jordan chain of length two. Moreover, we expand the operator $L $ in powers of $\eps$ which relies on expansions of the solution $u$:
\begin{align*}
	L &= L_0 + \eps L_1 + \mathcal{O}(\eps^2) \\
	L_0&=
	\begin{pmatrix}
		- \partial_x^2 + \zeta - 2 |u_0|^2 & - u_0^2 \\
		-\bar{u}_0^2 & - \partial_x^2 + \zeta - 2 |u_0|^2
	\end{pmatrix}, \
	L_1 =
	\begin{pmatrix}
	-2(u_0 \bar{u}_1 + \bar{u}_0 u_1) & -2 u_0 u_1 \\
	 -2 \bar{u}_0 \bar{u}_1 & -2 (u_0 \bar{u}_1 + \bar{u}_0 u_1)
	\end{pmatrix}, \\
	u &= u_0 + \eps u_1 + \mathcal{O}(\eps^2) , \qquad u_0 = \phi_{\theta_0}, \qquad u_1 = \iu \partial_\eps\theta(0) \phi_{\theta_0} - \iu \frac{f}{\zeta} + \varphi_1,
\end{align*}
and the equation for $\varphi_1$ is found from differentiating \eqref{eq:correction} w.r.t.~$\eps$ at $\eps=0$ (see also part two of the proof of Lemma~\ref{lem:nonsing}),
\begin{align*}
	L_{\theta_0} \varphi_1 = 2 |\phi_{\theta_0}|^2 \partial_\eps u_\infty(0) + \phi_{\theta_0}^2 \partial_\eps\bar{u}_\infty(0) + \iu \phi_{\theta_0}, \qquad \varphi_1 \perp \iu \phi_{\theta_0}.
\end{align*}
Substituting the expansions into the spectral problem $JL V = \lambda V$ yields equations at order $\eps^{1/2}$ and $\eps$,
\begin{align*}
	L_0 V_1 = \lambda_1 J^{-1} V_0, \quad
	L_0 V_2 + L_1 V_0
	= \lambda_1 J^{-1}V_1 +  \lambda_2 J^{-1} V_0.
\end{align*}
Since $J^{-1} V_0 \perp \Ker(L_0)$, we find $V_1 = \lambda_1 L_0^{-1} J^{-1} V_0 + \alpha V_0$, $\alpha \in \C$. Inserting this into the second equation amounts to
\begin{equation}\label{eq:cond_ev}
	L_0 V_2 = - L_1 V_0 +  \lambda_1^2 J^{-1} L_0^{-1} J^{-1} V_0+ \alpha\lambda_1 J^{-1} V_0 + \lambda_2 J^{-1} V_0
\end{equation}
and by the Fredholm alternative, \eqref{eq:cond_ev} is solvable if and only if
$$
	- L_1 V_0 +  \lambda_1^2 J^{-1} L_0^{-1} J^{-1} V_0+ \alpha\lambda_1 J^{-1} V_0 + \lambda_2 J^{-1} V_0 \perp \Ker(L_0).
$$
Since $ J^{-1} V_0 \perp \Ker(L_0)$, this yields
\begin{equation}\label{eq:formula_ev}
	\lambda_1^2 \langle J^{-1} L_0^{-1} J^{-1} V_0 ,V_0\rangle_{L^2} = \langle L_1 V_0, V_0\rangle_{L^2}.
\end{equation}

\begin{Lemma}\label{lem:formula_ev}
Let $\eps,\delta>0$ be sufficiently small and $B_\delta(0) \subset \C$ be the ball of radius $\delta$ centered at $\lambda =0$. Then $B_\delta(0) \cap \sigma_d(\mathcal{L})$
consists of four distinct simple eigenvalues, given by
$$
	\pm \eps , \pm \sqrt{\eps} \sqrt{- \pi f \sqrt{2\zeta}\sin\theta_0 } + \mathcal{O}(\eps) \in \sigma_d(\mathcal{L}).
$$
In particular, if $\sin\theta_0>0$, we have one unstable eigenvalue $\lambda = \eps$ of $\mathcal{L}$, which corresponds to a simple zero eigenvalue of $\mathcal{L} - \eps$, and two purely imaginary eigenvalues of $\mathcal{L}$:
$$
	\lambda_\pm =\pm \iu \left( \sqrt{\eps} \sqrt{\pi f \sqrt{2\zeta}\sin\theta_0} + \mathcal{O}(\eps)\right).
$$
If $\sin\theta_0< 0$, we have two real unstable eigenvalues of $\mathcal{L}$:
$$
	\lambda=\eps\quad\text{and}\quad\lambda_+=\sqrt{\eps} \sqrt{\pi f \sqrt{2\zeta}|\sin\theta_0|} + \mathcal{O}(\eps),
$$
where $\lambda_+$ is also an unstable eigenvalue of the operator $\mathcal{L} - \eps$ since it is of order $\mathcal{O}(\eps^{1/2})$.
\end{Lemma}
\begin{proof}
From the preceding discussion it remains to calculate the scalar products in \eqref{eq:formula_ev}. Let us start with $\langle J^{-1} L_0^{-1} J^{-1} V_0 ,V_0\rangle_{L^2}$. Consider the NLS
$$
	- \phi_{\theta_0}'' + \zeta \phi_{\theta_0} - |\phi_{\theta_0}|^2\phi_{\theta_0} = 0.
$$
Taking derivative w.r.t.~$\zeta$ yields
$$
	L_0 
	\begin{pmatrix}
		\partial_\zeta \phi_{\theta_0} \\ \partial_\zeta \bar\phi_{\theta_0}
	\end{pmatrix} =
	- 
	\begin{pmatrix}
		\phi_{\theta_0}	\\ \bar\phi_{\theta_0}
	\end{pmatrix}
	= J^{-1} V_0.
$$
The function $\partial_\zeta \phi_{\theta_0}$ is found from differentiating the formula of the NLS soliton and a straight forward calculation gives
\begin{align*}
	\langle J^{-1} L_0^{-1} J^{-1} V_0 ,V_0\rangle_{L^2} &= \int_\R \partial_\zeta \phi_{\theta_0} \bar\phi_{\theta_0} + \partial_\zeta \bar\phi_{\theta_0} \phi_{\theta_0} dx = 2 \zeta^{-1/2}.
\end{align*}
Next, we calculate the scalar product $\langle L_1 V_0, V_0\rangle_{L^2}$. We have
\begin{align*}
	\langle L_1 V_0, V_0\rangle_{L^2} = 2
	\int_\R 
	\begin{pmatrix}
		-\iu u_0^2 \bar u_1  \\ \iu \bar u_0^2 u_1
	\end{pmatrix} \cdot
	\overline{
	\begin{pmatrix}
		\iu u_0 \\ -\iu \bar u_0
	\end{pmatrix}} dx 
	= - 4 \RT \int_\R |u_0|^2 u_0 \bar u_1 dx. 
\end{align*}
Inserting the formula for $u_1$ into the integral yields
\begin{align*}
	&\RT \int_\R |u_0|^2 u_0 \bar u_1 dx  = 
	- \frac{f}{\zeta}  \|\phi_0\|_{L^3}^3 \sin\theta_0 + \RT \int_\R |\phi_{\theta_0}|^2 \phi_{\theta_0} \bar\varphi_1 dx \\
	&=  -\frac{f}{\zeta}  \|\phi_0\|_{L^3}^3 \sin\theta_0 
	+ \RT \int_\R |\phi_{\theta_0}|^2 \phi_{\theta_0} \overline{L_{\theta_0}^{-1} ( 2 |\phi_{\theta_0}|^2 \partial_\eps u_\infty(0) + \phi_{\theta_0}^2  \partial_\eps\bar{u}_\infty(0) + \iu \phi_{\theta_0})} dx \\
	& = - \frac{f}{\zeta}  \|\phi_0\|_{L^3}^3 \sin\theta_0 
	+\RT \int_\R L_{\theta_0}^{-1}(|\phi_{\theta_0}|^2 \phi_{\theta_0}) (\overline{ 2 |\phi_{\theta_0}|^2 \partial_\eps u_\infty(0) + \phi_{\theta_0}^2 \partial_\eps\bar{u}_\infty(0) + \iu \phi_{\theta_0}}) dx.
\end{align*}
Using that $L_{\theta_0} \phi_{\theta_0} = - 2|\phi_{\theta_0}|^2 \phi_{\theta_0}$ gives
\begin{align*}
	&\RT \int_\R L_{\theta_0}^{-1}(|\phi_{\theta_0}|^2 \phi_{\theta_0}) (\overline{ 2 |\phi_{\theta_0}|^2 \partial_\eps u_\infty(0) + \phi_{\theta_0}^2 \partial_\eps\bar{u}_\infty(0) + \iu \phi_{\theta_0}}) dx \\
	&= -\frac{1}{2} \RT \int_\R \phi_{\theta_0}( \overline{ 2 |\phi_{\theta_0}|^2 \partial_\eps u_\infty(0) + \phi_{\theta_0}^2 \partial_\eps\bar{u}_\infty(0) + \iu \phi_{\theta_0}}) dx \\
	&= \frac{3f}{2 \zeta} \|\phi_0\|_{L^3}^3\sin\theta_0.
\end{align*}
Finally, since $\|\phi_0\|_{L^3}^3 = \sqrt{2} \zeta \pi$, we have
$$
	\langle L_1 V_0, V_0\rangle_{L^2} = - 2 \frac{f}{\zeta} \|\phi_0\|_{L^3}^3 \sin\theta_0 
	= - 2\sqrt{2} f \pi\sin\theta_0  
$$
and thus from \eqref{eq:formula_ev} we obtain the desired formula for the simple eigenvalues
$$
	\lambda_\pm = \pm \sqrt{\eps} \sqrt{-\pi f \sqrt{2\zeta}\sin\theta_0} + \mathcal{O}(\eps),
$$
where in the case $\sin\theta_0>0$ the $\mathcal{O}(\eps)$ remainder is purely imaginary because of the Hamiltonian symmetry of the spectrum.
\end{proof}

Lemma~\ref{lem:formula_ev} proves the spectral instability of the wave $u$ if $\sin\theta_0<0$. However, if $\sin\theta_0>0$, no unstable eigenvalues of $\mathcal{L}-\eps$ occur from the splitting of the zero eigenvalue. Instead, we find a pair of purely imaginary eigenvalues $\lambda_\pm \in \iu\R$ of $\mathcal{L}$. Hence, we now focus on the case $\sin\theta_0>0$ and show that the only unstable eigenvalue of $\mathcal{L}$ is given by $\lambda = \eps$, which  then proves the spectral stability of the wave $u$. For this purpose, we employ the instablity index count developed in \cite{IndexCount1,IndexCount2} and also in \cite{ChugunovaPelinovsky}. To apply the instability count, we need to transform the spectral stability problem \eqref{eq:spec_stab_prob} into a problem with real-valued coefficients. Using similarity transformations with the matrizes
$$
	T_1 = \frac{1}{2}
	\begin{pmatrix}
		1 & 1 \\
		-\iu & \iu
	\end{pmatrix},\quad
	T_2 =
	\begin{pmatrix}
		\cos\theta_0 & \sin\theta_0 \\
		-\sin\theta_0 & \cos\theta_0 
	\end{pmatrix}
$$
such that $\tilde J = T_2T_1 J T_1^{-1}T_2^{-1}$, $\tilde L = T_2T_1 L T_1^{-1}T_2^{-1}$ we obtain the equivalent problem
$$
	\lambda \tilde V = (\tilde{J}\tilde{L} -\eps) \tilde V,
$$
and the eigenfunctions are related by $\tilde V = T_2T_1V$.
The real-valued operators $\tilde J,\tilde L$ are then of the form 
$$
	\tilde{J} = 
	\begin{pmatrix}
		0 & 1 \\ -1 & 0
	\end{pmatrix}, \qquad 
	\tilde{L} = 
	\begin{pmatrix}
		- \partial_x^2 + \zeta - 3 \phi_0^2  & 0\\
		0 & - \partial_x^2 + \zeta - \phi_0^2
	\end{pmatrix}
	+ \mathcal{O}(\eps).
$$
We recall necessary notation from \cite{KapProm}:
\begin{itemize}
	\item $n(A)$ denotes the \emph{number of negative eigenvalues} (counting multiplicities) of a linear operator $A$.
	\item Let $\lambda \in \sigma_d(\tilde{J}\tilde{L}) \cap (\iu \R \setminus\{0\})$ be an eigenvalue with algebraic multiplicity $m_a(\lambda)$ and $\Ker(\bigcup_{n \in \N}(\tilde{J}\tilde{L} - \lambda )^{n}) =\Span \{v_1,\dots,v_{m_a(\lambda)}\}$ be the generalized eigenspace. Then the \emph{negative Krein index} of $\lambda$ is defined by
	$$
		k_i^-(\lambda) := n (H), \quad\text{with}\quad H_{ij} := \langle \tilde{L} v_i ,v_j \rangle,
	$$
	and the total negative index is defined by
	$
		k_i^- := \sum_{\lambda \in  \sigma_d(\tilde{J}\tilde{L}) \cap \iu \R \setminus\{0\}} k_i^-(\lambda).
	$
	\item The \emph{real Krein index} is defined by
	$
		k_r := \sum_{\lambda \in \sigma_d(\tilde{J}\tilde{L}) \cap (0,\infty)} m_a(\lambda).
	$
	\item The \emph{complex Krein index} is defined by
	$
		k_c := \sum_{\lambda \in  \sigma_d(\tilde{J}\tilde{L}), \RT(\lambda)>0, \IT(\lambda)\not = 0} m_a(\lambda).
	$
\end{itemize}

Applying the instability index counting theory from \cite{ChugunovaPelinovsky,IndexCount1,IndexCount2} yields for $\eps>0$ sufficiently small the formula
\begin{align}\label{eq:counting_formula}
	k_r+k_i^- +k_c = n(\tilde{L}).
\end{align}

\begin{Remark}
In \cite{IndexCount1,IndexCount2} there appears an additional number $n(D)$ in the formula \eqref{eq:counting_formula}, which accounts for a nontrivial kernel generated by various symmetries of the problem. However, in Lemma~\ref{lem:formula_ev} we proved that $\Ker(\tilde{L}) = \{0\}$ provided $\eps>0$ is sufficiently small explaining the absence of this number in our situation (see also \cite{ChugunovaPelinovsky}).
\end{Remark}

In the next lemma, we use \eqref{eq:counting_formula} to show that $\eps \in \sigma(\mathcal{L})$ is the only unstable eigenvalue of $\mathcal{L}$ proving that $\sigma_d(\mathcal{L}-\eps) \subset \{-2\eps\} \cup \{\RT = - \eps\} \cup \{0\}$ with simple eigenvalues $\lambda=0,-2\eps$.

\begin{Lemma}\label{lem:index_calc}
Let $\eps>0$ be sufficiently small and $\sin\theta_0>0$. Then, $n(\tilde{L})= 3$, $k_r = 1$, and $k_i^- = 2$.
\end{Lemma}
 
\begin{proof}
$k_r \geq 1$ is clear since $\eps \in \sigma_d(\tilde{J}\tilde{L})$ due to the translational symmetry. Now consider the  Sturm-Liouville operators on the line $\R$,
$$
	L_+ := -\partial_x^2 + \zeta - 3 \phi_0^3, \quad L_-:= -\partial_x^2 + \zeta - \phi_0^3.
$$
We have $\Ker(L_+) = \Span\{\phi_0'\}$, $\Ker(L_-) = \Span\{\phi_0\}$ and since $\phi_0'$ has one zero and $\phi_0>0$ on $\R$ we find $n(L_+) = 1$, $n(L_-) = 0$ by the standard theory for Sturm-Liouville operators. In particular, we obtain
$$
	n \Big( \begin{pmatrix}
		L_+ & 0 \\ 0 & L_-
	\end{pmatrix} \Big) = 1, 
	\quad 0 \in \sigma_d \Big(
	\begin{pmatrix}
		L_+ & 0 \\ 0 & L_-
	\end{pmatrix}\Big), \quad m_a(0) = 2,
$$ 
and by means of perturbation theory for eigenvalues we conclude that there are at most three negative eigenvalues of $\tilde{L}$ for $\eps$ sufficiently small, which proves $n(\tilde{L}) \leq 3$. Moreover, by Lemma~\ref{lem:formula_ev} we find purely imaginary simple eigenvalues $\lambda_\pm\in\iu \R$ with $\bar{\lambda}_+=\lambda_-$ by the symmetry of the spectrum and we show that $k_i^-(\lambda_\pm) = 1$. Indeed, from Lemma~\ref{lem:formula_ev} and the discussion before, we have Puiseux expansions for the eigenvalue and corresponding eigenfunction
$$
	\lambda_+ = \sqrt{\eps} \lambda_1 + \mathcal{O}(\eps), \quad 
	\tilde V = \begin{pmatrix} 0 \\ \phi_0 \end{pmatrix} - \sqrt{\eps} \left[ \lambda_1 \begin{pmatrix} L_+^{-1}\phi_0 \\ 0\end{pmatrix} + \alpha 
	\begin{pmatrix}
		0 \\ \phi_0
	\end{pmatrix} \right]
	+ \mathcal{O}(\eps)
$$
where the eigenfunction is found from the relation $\tilde{V} = T_2T_1 V$. Hence direct calculations yield
\begin{align*}
	\langle \tilde{L} \tilde V,\tilde V \rangle_{L^2} = - \lambda_1 \langle \tilde{J}\tilde V,\tilde V \rangle_{L^2}
	=\eps 2 |\lambda_1|^2 \langle  L_+^{-1} \phi_0 ,\phi_0 \rangle_{L^2} + \mathcal{O}(\eps^{3/2}).
\end{align*}
Since $\langle L_+^{-1}\phi_0 , \phi_0 \rangle_{L^2} =- \langle \partial_\zeta\phi_0 , \phi_0 \rangle_{L^2}<0$, cf.~the proof of Lemma~\ref{lem:formula_ev}, we find $k_i^-(\lambda_+) = 1$ and from $k_i^-(\lambda_+)=k_i^-(\bar{\lambda}_+)$ and $\bar{\lambda}_+=\lambda_-$ it follows that $k_i^- \geq 2$. Thus, using \eqref{eq:counting_formula} we infer $n(L)=3$, $k_r=1$, and $k_i^- = 2$, which finishes the proof.
\end{proof}

Combing our results on the discrete and essential spectrum we finally obtain
$$
	\sigma(\mathcal{L}-\eps)\subset \{-2\eps\} \cup \{z \in \C: \ \RT z = -\eps \} \cup \{0\}
$$
with simple eigenvalues $\lambda = -2 \eps, 0$. Thus, the wave is spectrally stable provided that $\eps,\sin\theta_0>0$ as claimed.

\subsection{Proof of Theorem~\ref{thm:nonlin_stab}}\label{sec:proof_nonlin_stab}
We prove the asymptotic orbital stability in $H^1$ of the spectrally stable solitary waves of Theorem~\ref{thm:spec_stab}. The strategy of the proof follows the work in \cite{StanislavovaStefanov2018} where asymptotic orbital stability is obtained for \emph{periodic} spectrally stable solutions of LLE. Our method deviates from \cite{StanislavovaStefanov2018} when establishing uniform resolvent bounds. Indeed, high frequency resolvent estimates in \cite{StanislavovaStefanov2018} are only proven for the linearization operator with periodic coefficients. Using an abstract functional analytic approach, we extend this result to the case of localized perturbations.

\subsubsection*{\bf Linearized stability}
Let $u \in \C + H^2(\R)$ be a spectrally stable solution of \eqref{LLE_perp_stationary} for $\eps>0$ sufficiently small and denote by $\mathcal{L}-\eps$ the linearization about $u$.
In a first step we prove linearized stability. Note that decay of the semigroup $\left(\eu^{(\mathcal{L}-\eps)t}\right)_{t \geq 0}$ cannot be concluded immediately from the spectral stability of $\mathcal{L}-\eps$, since the spectrum of the operator is not confined to a sector of $\C$ and therefore the Spectral Mapping Theorem is a-priori not available. However, we can use the following characterization of exponential stability of semigroups in Hilbert spaces called the Pr\"uss-Theorem, cf.~\cite{Pruess} Corollary~4.

\begin{Theorem*}[Pr\"uss, \cite{Pruess}, Corollary~4]
Let $A$ be the generator of a $C_0$-semigroup $(\eu^{At})_{t \geq 0}$ in a Hilbert space $H$. Then $(\eu^{At})_{t \geq 0}$ is exponentially stable if and only if
$$
	\{\lambda\in\C :\ \RT(\lambda) \geq 0\} \subset \rho(A) \quad\text{and}\quad \sup\left\{
	\|(A-\lambda)^{-1}\|_{H \to H}:\ \lambda \in \C, \RT(\lambda)\geq 0 \right\} < \infty.
$$
\end{Theorem*}

Recall, that by the presence of the translational symmetry, $0 \in \sigma(\mathcal{L} - \eps)$, which violates the spectral condition in Pr\"uss Theorem. To overcome this problem, we introduce the spectral projection $P_0$ onto $\Ker(\mathcal{L} - \eps)$ and show that the restricted operator $(\mathcal{L} - \eps)|_E$ satisfies the conditions in Pr\"uss Theorem, where $E:=\Ker(P_0)$. This then leads to decay of the semigroup restricted to the subspace $E$, which is enough to establish the orbital stability result (cf.~Theorem~4.3.5~in~\cite{KapProm}). 

Recall the basic properties of the spectral projection: $(\mathcal{L} - \eps)P_0 =P_0 (\mathcal{L} - \eps) = 0$ and
\begin{align}\label{eq:spec_stabelized}
	\sigma((\mathcal{L} - \eps)|_E) = \sigma(\mathcal{L} - \eps) \setminus\{0\} \subset \{\lambda \in \C:\ \RT(\lambda) \leq -\eps\}.
\end{align}

\begin{Lemma}\label{lem:lin_exp_stability} 
There exist constants $\eta>0,C\geq 1$ such that
$$
	\|\eu^{(\mathcal{L} - \eps)t}|_E\|_{H^1 \to H^1} \leq C \eu^{-\eta t} \quad \text{for } t \geq 0.
$$
\end{Lemma}

\begin{proof}
It follows from the Lumer-Phillips Theorem (Theorem~II.3.15 in \cite{Engel_Nagel}) and the Bounded Perturbation Theorem (Theorem~III.1.3 in \cite{Engel_Nagel}) that $\mathcal{L} - \eps$ is the generator a $C_0$-semigroup on $(H^1(\R))^2$ and the same holds after restricting the operator to $E=\Ker(P_0)$. According to the Pr\"uss Theorem and \eqref{eq:spec_stabelized} the claim of the lemma follows if we show the uniform resolvent bound
$$
	\exists C>0: \quad
	\sup_{\RT(\lambda)\geq 0}\|(\mathcal{L} - \eps -\lambda)^{-1}(I-P_0)\|_{H^1 \to H^1} \leq C.
$$
This estimate is obtained in two steps.

\medskip

\emph{Step 1 (Uniform bound in $L^2$):} We show uniform resolvent estimates for the operator $(\mathcal{L} - \eps -\lambda)^{-1}(I-P_0):L^2 \to L^2$.
First note that the Hille-Yosida Theorem (\cite{Engel_Nagel} Theorem~3.8) ensures existence of constants $\gamma_1, C' > 0$ such that 
$$
	\sup_{\RT(\lambda)\geq \gamma_1}\|(\mathcal{L} - \eps -\lambda)^{-1}\|_{L^2 \to L^2} \leq C'.
$$
Moreover, using Theorem~\ref{thm:ex_bdd_resolvent} and Remark~\ref{rem:abst_resolvent_estimate} with $H= L^2(\R), A_\pm = -\partial_x^2 + \zeta - 2|u|^2, B = -u^2$ we find constants $\gamma_2, C''>0$ such that
$$
	\sup_{\RT(\lambda)\geq 0, |\IT(\lambda)|\geq \gamma_2}\|(\mathcal{L} - \eps -\lambda)^{-1}\|_{L^2 \to L^2} \leq C''.
$$
Finally, observe that $\lambda \mapsto (\mathcal{L} - \eps -\lambda)^{-1}(I-P_0)$ is an analytic function on $\{\lambda\in \C:\ \RT(\lambda) \geq 0\}$ and hence uniformly bounded on the compact set $\{\lambda \in \C:\ 0 \leq \RT(\lambda) \leq \gamma_1, |\IT(\lambda)|\leq \gamma_2\}$ with a bound $C''' \geq 0$. Thus for $C:= \max\{C',C'',C'''\}$ we have
$$
	\sup_{\RT(\lambda)\geq 0}\|(\mathcal{L} - \eps -\lambda)^{-1}(I-P_0)\|_{L^2 \to L^2} \leq C.
$$

\medskip

\emph{Step 2 (Uniform bound in $H^1$):} First, we establish a uniform bound in $H^2$ and then use an interpolation argument to find the desired $H^1$ bound. Note that there exist constants $c, \gamma \gg 1$ such that
$$
	\forall V = (v_1,v_2)^T \in (H^2(\R))^2: \quad \|V\|_{H^2} \leq c \|(L + J \eps + \gamma)V\|_{L^2}, \quad \|(JL - \eps )V\|_{L^2} \leq c\|V\|_{H^2},
$$
where we recall the decomposition $\mathcal{L} = JL$ from \eqref{def_lin_op}.
Hence we find for $\RT(\lambda)\geq 0,V \in H^2$,
\begin{align*}
	\|(JL &- \eps -\lambda)^{-1}(I -P_0)V\|_{H^2} \\
	&\leq c \|(L + J \eps + \gamma)(JL - \eps -\lambda)^{-1}(I-P_0)V\|_{L^2} \\
	&\leq c \|(JL - \eps)(JL - \eps -\lambda)^{-1}(I-P_0)V\|_{L^2} \\
	& \quad+ c \gamma \| (JL - \eps -\lambda)^{-1}(I-P_0)V\|_{L^2} \\
	&\leq c \|(JL - \eps -\lambda)^{-1}(I-P_0)(JL - \eps)V\|_{L^2} 
	+ c \gamma C \|V\|_{H^2} \\
	&\leq cC (c+\gamma) \|V\|_{H^2},
\end{align*}
by step 1, which implies the uniform bound
$$
	\sup_{\RT(\lambda)\geq 0}\|(JL - \eps -\lambda)^{-1}(I-P_0)\|_{H^2 \to H^2} \leq cC (c+\gamma).
$$
Interpolation of both estimates according to \cite{Lunardi} Theorem 2.6 yields
$$
	\sup_{\RT(\lambda)\geq 0}\|(JL - \eps -\lambda)^{-1}(I-P_0)\|_{H^1 \to H^1} \leq C
$$
for some constant $C>0$ and thus the claim follows.
\end{proof}

We still need to prove the uniform resolvent estimate for $\lambda \in \C$ with $\RT \lambda\geq 0$, $|\IT \lambda| \gg 1$ and conclude the nonlinear stability.

\subsubsection*{\bf High frequency resolvent estimates}
We establish uniform resolvent estimates for a family of operators in Hilbert spaces which generalizes the linearization operator $\mathcal{L}-\eps$. Our proof relies on techniques from \cite{GaeblerStanislavova2021} Section~3 where uniform resolvent estimates for NLS are considered. Similar resolvent estimates can also be found in \cite{BengelPelinovsky2023,PromislowKutz2000,StanislavovaStefanov2018}.

Let $H$ be a complex Hilbert space with scalar product $\langle \cdot, \cdot \rangle$ and norm $\|\cdot\| = \sqrt{\langle \cdot, \cdot \rangle}$. On $H \times H$ we consider the spectral problem
\begin{align}\label{eq:spec_prob}
	\left[
	\begin{pmatrix}
		-\iu & 0 \\ 0 & \iu
	\end{pmatrix}
	\begin{pmatrix}
		A_+ & B \\
		B^* & A_-
	\end{pmatrix}
	- \lambda
	\begin{pmatrix}
		I & 0 \\ 0 & I
	\end{pmatrix}
	\right]
	\begin{pmatrix}
		\phi_1 \\ \phi_2	
	\end{pmatrix}
	=
	\begin{pmatrix}
		\psi_1 \\ \psi_2
	\end{pmatrix},
\end{align}
where $\lambda = \lambda_r + \iu \lambda_i \in \C$ is a spectral parameter, $A_\pm: D \subset H \to H$ are closed self-adjoint linear operators with common domains $D = \Dom(A_\pm)$ which are either both bounded from below or from above by the bound $\gamma \in \R$, $B:H \to H$ is a bounded linear operator, $I:H \to H$ is the identity, $\phi_1,\phi_2 \in D$ and $\psi_1,\psi_2 \in H$. Under these assumptions the following theorem on uniform high-frequency resolvent estimates holds.

\begin{Theorem}\label{thm:ex_bdd_resolvent}
There exists $\rho = \rho(\gamma,\|B\|_{H \to H})>0$ such that for all $\lambda = \lambda_r + \iu \lambda_i \in \C$ with $|\lambda_i| \geq \rho$ and $\lambda_r \not = 0$ we have that for every given $(\psi_1,\psi_2) \in H \times H$ the spectral problem \eqref{eq:spec_prob} has a unique solution $(\phi_1, \phi_2) \in D \times D$ such that
$$
	\|\phi_1\| + \|\phi_2\| \lesssim |\lambda_r|^{-1} (\|\psi_1\| + \|\psi_2\|).
$$
\end{Theorem}

\begin{Remark}\label{rem:abst_resolvent_estimate}
Theorem~\ref{thm:ex_bdd_resolvent} gives a uniform resolvent estimate for the linear operator
$$
	\begin{pmatrix}
		-\iu & 0 \\ 0 & \iu
	\end{pmatrix}
	\begin{pmatrix}
		A_+ & B \\
		B^* & A_-
	\end{pmatrix}: D \times D \to H \times H
$$
for high frequencies $\lambda \in \C$ such that $|\IT(\lambda)| \gg 1$, $\RT(\lambda) \not =0$. For $\RT(\lambda) = 0$ we cannot expect such an estimate to be true, since the intersection of the spectrum and the imaginary axis is typically non-empty.
\end{Remark}

\begin{Remark}
Theorem~\ref{thm:ex_bdd_resolvent} can be applied to different variants of LLE. Indeed, consider the extended LLE
$$
	\iu u_t =  \sum_{k =1}^{2n} d_k \left(\iu \partial_x \right)^k u + (\zeta(x) - \iu \mu) u - |u|^2 u  + \iu f(x),\quad (x,t) \in \Omega\times\R, 
$$
where $\Omega \in \{\R, \R / \Z \}$, $n \in \N$, $d_{2n}\not = 0$, $d_{2n-1},\dots,d_1 \in \R$, $\zeta \in L^\infty(\Omega,\R)$, $f \in H^2(\Omega,\C)$, and $\mu \geq 0$. A similar equation is studied in \cite{BengelPelinovsky2023,Gasmi_Kirn,Gelens2008}. Then, the linearization $\mathcal{L}: (H^{2n}(\Omega))^2 \to (L^2(\Omega))^2$ about a stationary solution $u=u(x)$ reads
$$
	\mathcal{L}:=
	\begin{pmatrix}
		- \iu & 0 \\ 0 & \iu
	\end{pmatrix}
	\begin{pmatrix}
		\sum_{k =1}^{2n} d_k \left(\iu \partial_x \right)^k + \zeta(x) - 2 |u|^2 & -u^2 \\
		-\bar{u}^2 & \sum_{k =1}^{2n} d_k \left(-\iu \partial_x \right)^k + \zeta(x) - 2 |u|^2
	\end{pmatrix} - \mu.
$$
If we set $A_\pm = \sum_{k =1}^{2n} d_k \left(\pm\iu \partial_x \right)^k + \zeta(x) - 2 |u|^2$, $B = -u^2$, and replaced $\lambda$ by $\lambda+\mu$, the associated spectral problem fits into the framework of Theorem~\ref{thm:ex_bdd_resolvent}. In particular, the uniform resolvent estimates can be used to study the dynamics of the extended LLE close to stationary waves.
\end{Remark}

We need the following property of self-adjoint operators.

\begin{Lemma*}[\cite{Kato}, Chapter~5, Section~5]
Let $A : \Dom(A) \subset H \to H$ be a self-adjoint operator and $\lambda \in \rho(A)$ be in the resolvent set of $A$. Then
$$
	\|(A-\lambda)^{-1}\|_{H \to H} = \frac{1}{\textup{dist}(\sigma(A),\lambda)}.
$$
\end{Lemma*}

\begin{proof}[Proof of Theorem~\ref{thm:ex_bdd_resolvent}.]
We follow the strategy in \cite{GaeblerStanislavova2021} Section~3. Let us assume that $A_\pm$ are both bounded from below, i.e., $A_\pm \geq -\gamma$ for $\gamma>0$. The proof for the case that $A_\pm$ is bounded from above is similar. We write the spectral problem \eqref{eq:spec_prob} as the system
\begin{align}\label{eq:spec_sys}
	\left\{
	\begin{array}{l}
		(A_+ + \lambda_i-\iu \lambda_r) \phi_1 + B  \phi_2  = \psi_1, \\
		(A_- - \lambda_i+\iu \lambda_r) \phi_2 + B^*  \phi_1 = \psi_2,
	\end{array}\right.
\end{align}
where we have replaced $\iu\psi_1$ by $\psi_1$ and  $-\iu\psi_2$ by $\psi_2$. Now, consider the case $\lambda_i> 0$. By the previous lemma we infer that for all $\lambda_i \geq 2\gamma$ we find $-\lambda_i, -\lambda_i + \iu \lambda_r \in \rho(A_+)$ and
$$
	\|(A_+ + \lambda_i)^{-1}\|_{H \to H}, \|(A_+ + \lambda_i - \iu \lambda_r)^{-1}\|_{H \to H} \leq \frac{2}{\lambda_i}, \quad \|(A_+ + \lambda_i - \iu \lambda_r)^{-1}\|_{H \to H} \leq \frac{1}{|\lambda_r|}.
$$
In particular, we can solve the first equation in \eqref{eq:spec_sys} for $\phi_1$:
$$
	\phi_1 =- (A_+ + \lambda_i - \iu \lambda_r)^{-1} B \phi_2 + (A_+ + \lambda_i - \iu \lambda_r)^{-1} \psi_1
$$
and substituting the expression for $\phi_1$ into the second equation amounts to
\begin{align*}
	(A_- - \lambda_i + \iu \lambda_r ) \phi_2 - B^* (A_+ + \lambda_i - \iu \lambda_r)^{-1} B \phi_2
	= -B^* (A_+ + \lambda_i - \iu \lambda_r)^{-1} \psi_1 + \psi_2.
\end{align*}
By the resolvent identity we find
\begin{align*}
	(A_+ + \lambda_i - \iu \lambda_r)^{-1} - (A_+ + \lambda_i )^{-1}
	= \iu \lambda_r (A_+ + \lambda_i - \iu \lambda_r)^{-1}(A_+ + \lambda_i )^{-1}
\end{align*}
and consequently
\begin{align*}
	\big(A_- - B^* (A_+ + \lambda_i )^{-1} B - \lambda_i + \iu \lambda_r \big) \phi_2
	- \iu \lambda_r  B^* (A_+ + \lambda_i - \iu \lambda_r)^{-1}(A_+ + \lambda_i )^{-1} B\phi_2 \quad  \\
	= - B^* (A_+ + \lambda_i - \iu \lambda_r)^{-1} \psi_1 + \psi_2.
\end{align*}
The operator $B^* (A_+ + \lambda_i )^{-1} B :H \to H$ is bounded and symmetric and thus
$$
	\mathcal{A} := A_- - B^* (A_+ + \lambda_i )^{-1} B : D \to H
$$
is self-adjoint. From the previous lemma on resolvent bounds of self-adjoint operators we have $\|(\mathcal{A} - \lambda_i + \iu \lambda_r)^{-1}\|_{H \to H} \leq |\lambda_r|^{-1}$. Hence, we infer
\begin{align*}
	&\big[I - \iu \lambda_r (\mathcal{A} - \lambda_i + \iu \lambda_r)^{-1}  B^* (A_+ + \lambda_i - \iu \lambda_r)^{-1}(A_+ + \lambda_i )^{-1} B\big] \phi_2 \\
	&\quad = -(\mathcal{A} - \lambda_i + \iu \lambda_r)^{-1} B^* (A_+ + \lambda_i - \iu \lambda_r)^{-1} \psi_1 + (\mathcal{A} - \lambda_i + \iu \lambda_r)^{-1}\psi_2.
\end{align*}
Now observe that
\begin{align*}
	\|\iu \lambda_r (\mathcal{A} - \lambda_i + \iu \lambda_r)^{-1}  B^* (A_+ + \lambda_i - \iu \lambda_r)^{-1}(A_+ + \lambda_i )^{-1} B\|_{H \to H}
	\lesssim \lambda_i^{-2}.
\end{align*}
Thus for $\lambda_i>0$ sufficiently large the operator
$$
	I - \iu \lambda_r (\mathcal{A} - \lambda_i + \iu \lambda_r)^{-1}  B^* (A_+ + \lambda_i - \iu \lambda_r)^{-1}(A_+ + \lambda_i )^{-1} B: H \to H
$$
is invertible as a small perturbation of the identity and 
$$
	\Big\|\big[I - \iu \lambda_r (\mathcal{A} - \lambda_i + \iu \lambda_r)^{-1}  B^* (A_+ + \lambda_i - \iu \lambda_r)^{-1}(A_+ + \lambda_i )^{-1} B\big]^{-1} \Big\|_{H \to H} \leq \frac{1}{2}
$$
uniformly in $\lambda_i \gg 1$. For this reason and using $\|(\mathcal{A} - \lambda_i + \iu \lambda_r)^{-1}\|_{H \to H} \leq |\lambda_r|^{-1}$ we find
$$
	\|\phi_2\| \lesssim |\lambda_r|^{-1} (\|\psi_1\| + \|\psi_2\|)
$$
as well as
$$
	\|\phi_1\| \lesssim |\lambda_r|^{-1} (\|\psi_1\| + \|\psi_2\|)
$$
and both estimates are independent of $\lambda_i$ provided $\lambda_i \gg 1$. In summary for $\lambda = \lambda_r + \iu\lambda_i \in \C$ with $\lambda_r \not = 0$ and $\lambda_i \gg 1$ the resolvent in \eqref{eq:spec_prob} exists and is uniformly bounded in $\lambda_i$. In the same way one can show the existence and boundedness of the resolvent for $ \lambda_i \ll -1$ and the claim follows.
\end{proof}

\subsubsection*{\bf Asymptotical orbital stability}
The proof of the nonlinear stability follows as a direct consequence of the exponential decay estimate in Lemma~\ref{lem:lin_exp_stability} and Theorem~4.3.5 in \cite{KapProm} applied to \eqref{eq:dyn_pert}. Notice that the nonlinearity in \eqref{eq:dyn_pert} is locally Lipschitz, so that all assumptions of the theorem in \cite{KapProm} are satisfied. In conclusion, the solitary wave bifurcating from $\phi_{\theta_0}$ with $\sin\theta_0>0$ is asymptotically orbitally stable against localized perturbations in $H^1(\R)$ and the proof of Theorem~\ref{thm:nonlin_stab} is completed.

\section*{Acknowledgments} 
The author is grateful to W.~Reichel and B.~de Rijk for their valuable support during various stages of this work.
This research was funded by the Deutsche Forschungsgemeinschaft (DFG, German Research Foundation) -- Project-ID 258734477 -- SFB 1173.

\bibliographystyle{siam}	
\bibliography{bibliography}
	
\end{document}